\documentclass[hidelinks,onefignum,onetabnum]{siamart251216}
\usepackage{amsmath, amsfonts, amssymb, booktabs}
\usepackage{graphicx, xcolor, bm, soul}
\usepackage{algorithm}
\usepackage{algpseudocode}
\usepackage{url}
\usepackage{enumitem}

\newcommand{\C}{\mathbb{C}}
\newcommand{\R}{\mathbb{R}}

\title{Recovering Complex Unitary Eigenspaces from Real-Valued Embeddings\thanks{Submitted to the editors \today}}

\author{Stefanie Guenther\thanks{Lawrence Livermore National Laboratory, {\email{guenther5@llnl.gov}} }
\and N.~Anders Petersson\thanks{Lawrence Livermore National Laboratory, {\email{petersson1@llnl.gov}}}
}

\begin{document}

\maketitle

\begin{abstract}
We consider the problem of recovering a unitary eigendecomposition of a complex unitary matrix from that of its embedded real-valued formulation. Such formulations arise naturally in scientific computing workflows that employ real-arithmetic solvers by representing complex matrices in term of their real and imaginary parts.
While the reconstruction is trivial when the spectrum of the real-valued embedding is simple, degenerate and/or complex conjugated eigenvalues introduce ambiguities because each eigenspace may include contributions from both the unitary matrix and its complex conjugate. We prove that this ambiguity can always be resolved by applying a structured projection to the eigenspaces of the real-valued embedding, followed by a rank-revealing orthonormalization. The resulting procedure recovers the eigenvalues and an unitary eigenbasis for the original unitary matrix, with correct multiplicities of degenerate eigenvalues. 
\end{abstract}



\section{Introduction}
Unitary matrices arise in many areas of science and engineering, such as quantum physics~\cite{nielsen00}, signal processing~\cite{oppenheim2010discrete}, photonics~\cite{Reck-94}, and in numerical linear algebra~\cite{golub2013matrix}, to mention a few. The common thread in these applications is norm preservation: wherever a physical or mathematical transformation must conserve some notion of energy, probability, or information, the natural mathematical object is a unitary matrix.

The eigendecomposition of a unitary matrix $U$ can always be written as $U= V \Lambda V^\dagger$, where $V$ is unitary and the eigenvalues are complex numbers on the unit circle, $\Lambda = \text{diag}(e^{i\theta_1}, \ldots, e^{i\theta_n})$. This decomposition is very useful because it exposes the real-valued phases $\theta_k$ and the unitary eigenbasis $V$. For example, the eigendecomposition of unitary matrices is a natural tool for spectral analysis of periodic or quasi-periodic signals. Another example occurs in quantum physics, where the eigendecomposition of the unitary time-evolution operator $U(t) = \exp(-itH)$ can be used to calculate the Hamiltonian generator $H$ from the matrix logarithm, $\log(U) = V \log(\Lambda) V^\dagger$, subject to the usual branch ambiguities.

While modern distributed memory and large-scale sparse eigensolvers, such as those available in ARPACK~\cite{lehoucq1998arpack}, SLEPc~\cite{Hernandez:2005} or ScaLAPACK~\cite{scalapack97}, support both real and complex arithmetic, practical constraints frequently favor real-valued formulations. In particular, large scientific software stacks are often compiled and optimized for real-valued arithmetic. Switching to complex arithmetic may require recompilation of dependencies, introduce compatibility issues, or lead to performance degradation due to increased memory traffic or reduced efficiency of the underlying kernels. As a result, it is common in high-performance computing (HPC) applications to represent complex vectors and matrices by their real and imaginary parts and to operate on the equivalent real-valued system.

For a unitary matrix $U=A+iB\in\mathbb{C}^{n\times n}$, with $A,B\in\R^{n\times n}$, a real-valued formulation is given by the block matrix,
\begin{align} \label{eq:realvaluedMatrix}
    M = \begin{bmatrix} A & -B \\ B & A \end{bmatrix} \in \R^{2n \times 2n},
\end{align}
which acts on the stacked real and imaginary components of a complex vector. This formulation, referred to as the K1 formulation in the literature~\cite{Day2001ERF}, is widely used in practice. In particular, with respect to the induced Euclidean norm, the conditioning of the real-valued system matches that of the original complex formulation, allowing established real-valued solvers to be applied effectively, including preconditioning techniques tailored to the block-structured system~\cite{axelsson2014comparison}.

In this work, we address the corresponding eigendecomposition problem. Specifically, given an eigendecomposition of the real-valued orthogonal matrix $M$, how can one recover a unitary eigendecomposition of the original complex unitary matrix $U$? 
The answer is immediate when the spectrum of $M$ is simple. In this case, each eigenvector of $M$ corresponds directly to either an eigenvector of $U$, or of its complex-conjugate counterpart $\bar U$, and the desired eigenvectors can be identified by a simple projection. However, repeated eigenvalues introduce ambiguity. Degeneracies may arise from repeated eigenvalues of $U$, from real eigenvalues shared by $U$ and $\bar U$, or from complex-conjugate pairs of eigenvalues. In these cases, an eigensolver may return an arbitrary basis within each eigenspace of $M$, mixing contributions associated with $U$ and $\bar U$, making the reconstruction less obvious. The purpose of this work is to establish that the approach from the simple case can be extended to the degenerate case. We prove that by applying the projection to each eigenspace of the real-valued matrix $M$, followed by a rank-revealing orthonormalization, recovers exactly the eigenspaces of the original unitary matrix with correct multiplicities. The resulting procedure is straightforward to implement, requires no modification of existing real-valued eigensolvers, and is directly applicable in HPC workflows where complex arithmetic is undesirable or unavailable.

\section{Recovering the unitary eigendecomposition by projection}
Let $U\in\mathbb{C}^{n\times n}$ be a unitary matrix, i.e., $U^\dagger U = UU^\dagger = I_n$\footnote{The conjugate transpose of a matrix $A$ is denoted $A^\dagger \equiv \bar{A}^T$.}.
As a normal matrix, $U$ admits a unitary eigendecomposition $U=V\Lambda V^\dagger$ where $V\in \C^{n\times n}$ is unitary and 
$\Lambda=\mbox{diag}(\lambda_1, \dots, \lambda_n)$ contains the eigenvalues of $U$.
Consider the real-valued representation of $U$ given by the orthogonal (and therefore normal) matrix $M \in \R^{2n \times 2n}$ in \eqref{eq:realvaluedMatrix}. 
The following lemma characterizes the eigendecomposition of $M$ in terms of that of $U$. 
\begin{lemma}\label{lemma:eigendecomposition_M}
    Let $U=V\Lambda V^\dagger$ be a unitary eigendecomposition. Then the matrix
    $M$ defined in \eqref{eq:realvaluedMatrix} admits the eigendecomposition
    \begin{align}
        M = \left[ W_+ \quad W_-  \right] \begin{bmatrix}
            \Lambda & 0 \\ 0 & \bar \Lambda
        \end{bmatrix} 
        \begin{bmatrix} W_+ & W_-\end{bmatrix}^\dagger
    \end{align}
    where 
    \begin{align}\label{eq_W+W-}
        W_+ := 
            \frac{1}{\sqrt{2}} \begin{bmatrix}
                V  \\ -iV 
            \end{bmatrix} ,
         \quad W_- := 
            \frac{1}{\sqrt{2}} \begin{bmatrix}
                -i\bar V  \\ \bar V 
            \end{bmatrix} 
    \end{align}
    and $W:=\left[W_+ \quad W_- \right]$ is unitary. Hence, $\sigma(M) = \sigma(U)\cup \sigma(\bar U)$. 
\end{lemma}
\begin{proof}
A more general statement is provided in \cite{Freund1992CG}, proposition 5.1 a), and its proof is easily verified by noting that  
\[
W = Q\begin{bmatrix} V & 0 \\0 & \bar V \end{bmatrix},\quad 
Q = \frac 1{\sqrt{2}} \begin{bmatrix}
       I_n & -iI_n \\ -iI_n & I_n 
    \end{bmatrix},
\]
where $Q$ is unitary, and by evaluating $W \begin{bmatrix} \Lambda & 0\\ 0 & \bar \Lambda\end{bmatrix}W^\dagger = 
Q \begin{bmatrix}
    U & 0 \\ 0 & \bar U
\end{bmatrix}Q^\dagger = M.$
\end{proof}

We here consider the inverse problem. Suppose that a real-valued eigensolver has been applied to $M \in \R^{2n\times 2n}$, producing an eigendecomposition 
\begin{align}
 MZ = Z \Sigma, 
\end{align}
where $\Sigma  \in\mathbb{C}^{2n\times 2n}$ is diagonal and $Z\in \C^{2n\times 2n}$ contains a complete eigenbasis. 
We aim to construct a unitary matrix $V\in \C^{n\times n}$ and a diagonal matrix $\Lambda\in C^{n\times n}$ such that $U=V\Lambda V^\dagger$.
We remark that, when the real-valued formulation $M$ is passed to a general-purpose eigensolver operating in real-valued arithmetic, the resulting eigenbasis need not be unitary nor preserve the structural separation between the eigenspaces associated with $U$ and $\bar U$, even though $M$ is normal and therefore admits a unitary eigendecomposition. Here, we consider the general case where the given eigenbasis of $M$ is not orthonormal, and note that orthonormality does not simplify the procedure.

The reconstruction is straightforward when the spectrum of $M$ is simple with a set of $2n$ \textit{distinct} eigenvalues, i.e., if the multiplicity\footnote{$\sigma(M)$ denotes the spectrum of $M$, i.e., the set of eigenvalues of $M$. Since $M$ and $U$ are normal, the algebraic multiplicity of eigenvalues equals the geometric multiplicity, which we simply call the multiplicity, and denote by $m_M(\mu)$ and $m_U(\mu)$.} of each $\mu \in \sigma(M)$ is $m_M(\mu)=1$. In that case, each eigenspace is one-dimensional and, since $M$ is normal, the eigenvectors are mutually orthogonal. 
Each eigenvector $z_\mu$ then corresponds to the distinct eigenvalue $\mu\in\sigma(M)$, and
Lemma \ref{lemma:eigendecomposition_M} shows that it must be of one of two forms: 
\begin{align}
   z_\mu &= \alpha \begin{bmatrix} v \\ -iv \end{bmatrix} \quad \text{if $\mu\in\sigma(U)$}, \\
   \text{or} \quad z_\mu &=  \alpha \begin{bmatrix} -i\bar v \\ \bar v \end{bmatrix} \quad \text{if $\mu\in\sigma(\bar U)$},
\end{align}
where $v$ is an eigenvector of $U$ and $\alpha\neq 0$ is a scaling factor. 
Define the projection operator,
\begin{align}
    L = [I_n \quad iI_n] \in \C^{n\times 2n}, \qquad  z=\begin{bmatrix}
        z^{(1)} \\ z^{(2)} 
    \end{bmatrix} \mapsto Lz=z^{(1)} + iz^{(2)} \in \C^n.
\end{align} 
Applying $L$ to the eigenvector $z_\mu$ separates the two cases:
\begin{align}
    Lz_\mu &= \sqrt{2}\alpha\, v,  \quad \text{if $\mu\in\sigma(U)$},\\
    \text{or} \quad Lz_\mu &= 0,  \qquad\quad\, \text{if $\mu\in\sigma(\bar U)$}.
\end{align}
Thus, when the spectrum is simple, the eigenvectors of $U$ are recovered directly by projecting the eigenvectors of $M$ and discarding the zero results.

The situation is less straightforward in the presence of repeated eigenvalues in $M$. Degeneracy may arise because $U$ itself has repeated eigenvalues, because real eigenvalues are shared by both $U$ and $\bar U$, or because complex conjugate eigenvalues of $U$ causes overlap in the embedded spectrum. 
From Lemma \ref{lemma:eigendecomposition_M} we know that $\sigma(M) = \sigma(U) \cup \sigma(\bar{U})$, the multiplicity of each eigenvalue $\mu\in\sigma(M)$ therefore satisfies
    \begin{align}\label{eq_multiplicity}
        m_M(\mu) = m_U(\mu) + m_{\bar{U}}(\mu),
    \end{align}
    with the convention that $m_{\bar U}(\mu) = 0$ if $\mu\notin \sigma(\bar U)$.
As a simple example, if $U$ has a real eigenvalue $\mu \in \R$ with multiplicity $m_U(\mu)$, then $\mu$ will also be an eigenvalue of $\bar U$ with the same multiplicity, giving $m_M(\mu)=2m_U(\mu)$. In such cases, an eigensolver may return an arbitrary basis spanning the eigenspace associated with $\mu$ that mixes contributions from both the $U$- and $\bar U$-associated subspaces. 

The following lemma shows that the above projection strategy remains valid in the general setting. An additional rank-revealing orthonormalization within each eigenspace of $Z$ is required to determine the multiplicity of each eigenvalue and the orthonormal basis vectors in the unitary eigendecomposition of $U$. 
\begin{lemma}\label{lemma:exmu}
    Let $M=Z\Sigma Z^{-1}$ be an eigendecomposition of the real-valued matrix $M$ defined in \eqref{eq:realvaluedMatrix}, and let $Z_\mu\in \C^{2n\times m_M(\mu)}$ denote the eigenvectors associated with eigenvalue $\mu\in\sigma(M)$. Define 
    \begin{align}
        X_\mu := LZ_\mu \in \C^{n\times m_M(\mu)}.
    \end{align}
    Then, $\mbox{rank}(X_\mu) = m_U(\mu)$ and $X_\mu$ spans the eigenspace of $U$ associated with $\mu$ whenever $m_U(\mu) > 0$.  
\end{lemma}

\begin{proof}
By assumption, we have $MZ_\mu = \mu Z_\mu$ for each unique $\mu\in\sigma(M)$, and from Lemma \ref{lemma:eigendecomposition_M} we further know that 
\[
M[W_{+,\mu} \,W_{-,\mu}]  = \mu [W_{+,\mu} \, W_{-,\mu}],
\]
where the subscript $\mu$ collects those $m_M(\mu)$ columns of $W=[W_+\quad W_-]$ that are associated with the eigenvalue $\mu\in \sigma(M)$. 
Because $W$ is unitary and $\text{rank}(Z_\mu) = m_M(\mu)$, there exists an invertible coefficient matrix, 
\begin{align}
    S = \begin{bmatrix}
  S_{+} \\
  S_{-}
\end{bmatrix}\in \C^{m_M(\mu) \times m_M(\mu)},
\end{align}
such that
\begin{align}
   Z_{\mu} = [W_{+,\mu} \, W_{-,\mu}] \begin{bmatrix}
      S_{+} \\
      S_{-}
    \end{bmatrix} \Pi = (W_{+,\mu} S_+ + W_{-,\mu}S_-) \Pi,
\end{align}
for a permutation matrix $\Pi\in\mathbb{R}^{m_M(\mu)\times m_M(\mu)}$.
The number of columns in $W_{+,\mu}$ and $W_{-,\mu}$, and hence the number of rows of $S_+$ and $S_-$ are precisely the respective multiplicities $m_U(\mu)$ and $m_{\bar U}(\mu)$, where $m_U(\mu) + m_{\bar U}(\mu) = m_M(\mu)$ from \eqref{eq_multiplicity}. 

Now apply the projection operator $L$. Since 
$LW_- = 0$ and $LW_+ = \sqrt{2}V$, $L$ extracts the contribution of eigenvectors associated with $U$, while annihilating those corresponding to $\bar U$, 
\begin{align}
    X_{\mu} = LZ_{\mu} = \sqrt 2 V_{\mu}S_+ \Pi \quad \in \C^{n\times m_M(\mu)}.
\end{align}
If $m_U(\mu) = 0$, then the coefficient block $S_+ \in \C^{m_U(\mu)\times m_M(\mu)}$ is absent, implying that $X_\mu = 0\in \C^{n\times m_{\bar U}(\mu)}$ and hence $\mbox{rank}(X_\mu) = 0$. If $m_U(\mu) > 0$, then $S_+$ has full row rank $m_U(\mu)\geq 1$ because $S$ is non-singular, and, since the columns of $V_\mu$ are linearly independent, we conclude that $\mbox{rank}(X_\mu) = \mbox{rank}(V_\mu S_+) = m_U(\mu)$ and $X_{\mu}$ spans the same space as $V_{\mu}$.
\end{proof}

Lemma \ref{lemma:exmu} shows that $\mbox{range}(X_{\mu}) = \mbox{range}(V_\mu)$, which is zero if and only if $\mu\notin\sigma(U)$. A rank-revealing orthonormalization of $X_{\mu}$ therefore either recovers $m_U(\mu)\geq 1$ orthonormal eigenvectors when $\mu\in\sigma(U)$, or yields rank zero in which case $\mu\notin\sigma(U)$. Since the eigenspaces of a normal matrix corresponding to different eigenvalues are mutually orthogonal, 
orthonormalizing $X_{\mu} = LZ_{\mu}$ for each unique eigenvalue $\mu\in\sigma(M)$ yields an orthonormal eigenbasis for $U$, together with the correct eigenvalue multiplicities. 

Algorithm \ref{alg:evecs} summarizes the reconstruction of an orthonormal basis of the eigen\-space of $U$ from the eigendecomposition of $M$. The algorithm iterates over the unique eigenvalues $\mu\in\sigma(M)$, which first requires a grouping of the eigenvalues. 
Once the eigenspaces are grouped, the reconstruction consists of applying the projection $L$ and performing a rank-revealing orthonormalization within each projected subspace. 

\begin{algorithm} 
  \caption{Reconstruction of a unitary eigendecomposition of $U$ from $M$}\label{alg:evecs}
  \begin{algorithmic}[1]
    \State \textbf{Inputs:} Eigendecomposition $M=Z\Sigma Z^{-1}$ 
    \State \textbf{Outputs:} Unitary matrix $V$ and diagonal matrix $\Lambda$ such that $U=V\Lambda V^\dagger$
    \State Initialize $V \leftarrow [\,]$, $\Lambda \leftarrow [\,]$
    \For{each unique eigenvalue $\mu\in \sigma(M)$} 
        \State Extract eigenvectors $Z_{\mu} \in C^{2n \times m_M(\mu)}$
        \State Compute $X_\mu = LZ_{\mu} =  Z_{\mu}^{(1)} + iZ_{\mu}^{(2)}$
        \State Compute a rank-revealing orthonormalization of $X_{\mu}$ 
        \If{$\mbox{rank}(X_\mu) > 0$}
            \State Let $Q_{\mu}$ be an orthonormal basis for $\mbox{range}(X_\mu)$
            \State Append columns of $Q_{\mu}$ to $V$
            \State Append $\mu$ to diagonal entries of $\Lambda$ with multiplicity equal to $\mbox{rank}(X_\mu)$
        \EndIf
    \EndFor
  \end{algorithmic}
\end{algorithm}

\section{Conclusion}
We have shown how to recover an orthonormal eigendecomposition of a complex unitary matrix $U$ from the eigendecomposition of its real-valued embedding $M$. While the reconstruction is immediate when the spectrum is simple, degeneracies obscure the correspondence between the eigenspaces associated with $U$ and those associated with $\bar U$.
The main result establishes that this ambiguity can always be resolved by a simple projection,
followed by a rank-revealing orthonormalization within each eigenspace of $M$. The resulting procedure is straightforward to implement, compatible with existing real-valued eigensolvers, and provides a practical bridge between real-arithmetic computational workflows and spectral analysis of unitary matrices. 

\section*{Acknowledgments}
The authors gratefully acknowledge financial support from the Office of Science, Advanced Scientific Computing Research (ASCR) within the U.S.~Department Of Energy, award \#~SCW1895. The first author thanks Dr.~Mohammed Azeem Sheik for fruitful discussion during the development of this manuscript. This work was performed under the auspices of the U.S. Department of Energy by Lawrence Livermore National Laboratory under Contract DE-AC52-07NA27344. This is contribution LLNL-JRNL-2019410.  

\bibliographystyle{siamplain}
\bibliography{references.bib}

@article{Freund1992CG,
author = {Freund, Roland W.},
title = {Conjugate Gradient-Type Methods for Linear Systems with Complex Symmetric Coefficient Matrices},
journal = {SIAM Journal on Scientific and Statistical Computing},
volume = {13},
number = {1},
pages = {425-448},
year = {1992},
doi = {10.1137/0913023},
eprint = {https://doi.org/10.1137/0913023}
}

@article{Day2001ERF,
author = {Day, David and Heroux, Michael A.},
title = {Solving Complex-Valued Linear Systems via Equivalent Real Formulations},
journal = {SIAM Journal on Scientific Computing},
volume = {23},
number = {2},
pages = {480-498},
year = {2001},
doi = {10.1137/S1064827500372262}
}

@article{axelsson2014comparison,
  title={A comparison of iterative methods to solve complex valued linear algebraic systems},
  author={Axelsson, Owe and Neytcheva, Maya and Ahmad, Bashir},
  journal={Numerical Algorithms},
  volume={66},
  number={4},
  pages={811--841},
  year={2014},
  publisher={Springer}
}

@book{golub2013matrix,
  title={Matrix Computations},
  author={Golub, Gene H and Van Loan, Charles F},
  year={2013},
  publisher={Johns Hopkins University Press},
  address={Baltimore},
  edition={4th},
  isbn={978-1-4214-0794-4}
}

@book{nielsen00,
  author = {Nielsen, Michael A. and Chuang, Isaac L.},
  publisher = {Cambridge University Press},
  title = {Quantum Computation and Quantum Information},
  year = 2000
}

@article{Reck-94,
  title = {Experimental realization of any discrete unitary operator},
  author = {Reck, Michael and Zeilinger, Anton and Bernstein, Herbert J. and Bertani, Philip},
  journal = {Phys. Rev. Lett.},
  volume = {73},
  issue = {1},
  pages = {58--61},
  numpages = {0},
  year = {1994},
  month = {Jul},
  publisher = {American Physical Society},
  doi = {10.1103/PhysRevLett.73.58},
  url = {https://link.aps.org/doi/10.1103/PhysRevLett.73.58}
}

@book{oppenheim2010discrete,
  title={Discrete-Time Signal Processing},
  author={Oppenheim, Alan V. and Schafer, Ronald W.},
  isbn={978-0131988422},
  edition={3rd},
  year={2010},
  publisher={Pearson},
  address={Upper Saddle River, NJ}
}

@Article{Hernandez:2005,
  author  = "Vicente Hernandez and Jose E. Roman and Vicente Vidal",
  title   = "{SLEPc}: A scalable and flexible toolkit for the solution of eigenvalue problems",
  journal = "ACM Trans. Math. Software",
  volume  = "31",
  number  = "3",
  pages   = "351--362",
  year    = "2005",
  doi     = "10.1145/1089014.1089019"
}

@book{lehoucq1998arpack,
  author    = {Lehoucq, R. B. and Sorensen, D. C. and Yang, C.},
  title     = {ARPACK Users' Guide: Solution of Large-Scale Eigenvalue Problems with Implicitly Restarted Arnoldi Methods},
  publisher = {Society for Industrial and Applied Mathematics (SIAM)},
  year      = {1998},
  address   = {Philadelphia, PA},
  isbn      = {978-0-89871-407-4},
  url       = {https://epubs.siam.org/doi/book/10.1137/1.9780898719628}
}

@book{scalapack97,
  author = {L. S. Blackford and J. Choi and A. Cleary and E. D'Azevedo and 
            J. Demmel and I. Dhillon and J. Dongarra and S. Hammarling and 
            G. Henry and A. Petitet and K. Stanley and D. Walker and R. C. Whaley},
  title = {ScaLAPACK Users' Guide},
  publisher = {Society for Industrial and Applied Mathematics},
  year = {1997},
  address = {Philadelphia, PA},
  isbn = {0-89871-400-5},
  url = {http://www.netlib.org/scalapack/slug/},
  note = {Also available at \url{https://epubs.siam.org/doi/10.1137/1.9780898719642}}
}

\end{document}